\documentclass[a4paper, 12 pt]{amsart}
\usepackage{amsmath}
\usepackage{amsfonts}
\usepackage{amssymb}
\usepackage{amsthm}
\usepackage[all, cmtip]{xy}
\usepackage{tikz-cd}
\usepackage{mathrsfs}
\usepackage{textcomp}
\usepackage{hyperref}

\newcommand{\Xred}{X_{\mathrm{red}}}

\newcommand{\sI}{{\mathscr I}}
\newcommand{\sN}{{\mathscr N}}
\newcommand{\sO}{{\mathscr O}}

\renewcommand{\deg}[1]{\operatorname{deg}(#1)}
\newcommand{\Hom}{\operatorname{Hom}}

\theoremstyle{plain}
\newtheorem{thm}{Theorem}
\newtheorem{cor}{Corollary}
\newtheorem{fact}{Fact}
\newtheorem{question}{Question}

\theoremstyle{remark}
\newtheorem{rmk}{Remark}

\begin{document}

\title[On the compactified Jacobian of a ribbon]{On the irreducible components of the compactified Jacobian of a ribbon}

\author{Michele Savarese}
\address{Dipartimento di Matematica e Fisica,
Universit\`a Roma Tre, \\
Largo San Leonardo Murialdo 1,
00146 Roma (Italia)}
\email{msavarese@mat.uniroma3.it}

\keywords{Compactified Jacobians, Ribbons, Generalized Line Bundles, Multiple Curves}
\subjclass[2010]{14D20, 14H40, 14H60}

\begin{abstract}
In this paper we study the irreducible components of the compactified Jacobian of a ribbon $X$ of arithmetic genus $g$ over a smooth curve $\Xred$ of genus $\bar{g}$. We prove that when $g\geq 4\bar{g}-2$ the moduli space of rank $2$ semistable vector bundles over $\Xred$ is not an irreducible component and we determine the irreducible components in which it is contained. This answers a question of D. Chen and J.L. Kass in \cite{CK} and completes their results.
\end{abstract}

\maketitle

The aim of this short paper is to answer to \cite[Question 4.8]{CK}, so the conventions and notations are the same of the cited place. Let us recall them.

A \emph{curve} is an irreducible projective $k$-scheme of dimension $1$, where $k$ is an algebraically closed field; while a \emph{ribbon} is a curve whose reduced subscheme is a smooth curve and whose nilradical is locally generated by a nonzero and square-zero element. In other words it is a primitive multiple curve of multiplicity $2$ in the sense of \cite{DR}. Observe that this definition of ribbon is more restrictive than others present in literature; e.g. in \cite[\S I]{BE} irreducibility, projectiveness and dimension $1$ are not required and the reduced subscheme is just connected.

Throughout this article $X$ will denote a ribbon and $\sN\subset\sO_X$ its nilradical. 
It is immediate to check that $\sN$ can be seen as a line bundle on $\Xred$ and that its degree on $\Xred$ is $\deg{\sN}=2\bar{g}-1-g$, where $\bar{g}$ is the genus of $\Xred$ and $g=1-\chi(\sO_X)$ is the \emph{genus} of $X$.

A \emph{generalized line bundle} $\sI$ on $X$ is a pure (i.e. not having proper subsheaves of finite support) coherent sheaf whose generic stalk $\sI_{\eta}$ is isomorphic to $\sO_{X,\eta}$, where $\eta$ is the generic point of $X$ (in other words it is a generalized divisor in the sense of Hartshorne, because on primitive multiple curves pure is equivalent to reflexive). Its \emph{degree} is $\deg{\sI}=\chi(\sI)-\chi(\sO_X)$, while its \emph{index} $b(\sI)$ is the length of the torsion part of its restriction to $\Xred$. This is the definition of index given for any pure coherent sheaf on a ribbon in \cite[\S 6.3.7]{DR}, which is equivalent to the more involved one given in \cite[Definition 2.7]{CK}, that is specific for generalized line bundles. It is easy to check that $b(\sI)$ is a non-negative integer. Moreover, it holds also that $\deg{\sI}-b(\sI)$ is an even number; more precisely $\deg{\sI}-b(\sI)=2\deg{\overline{\sI}}$, where $\overline{\sI}$ is the unique line bundle on $\Xred$ such that $\sI|_{\Xred}=\overline{\sI}\oplus T$ with $T$ a torsion sheaf (see \cite[Fact 2.8]{CK} and its reference; note that in the statement of the cited fact there is a typographical error: $b(\sI')$ should be $\deg{\overline{\sI}}$).

Chen and Kass study the moduli space $\mathrm{M}(\sO_X,P_d)$ of semistable pure coherent sheaves of generic length $2$ on $X$ of fixed degree $d$. They show that the points of $\mathrm{M}(\sO_X,P_d)$ correspond either to $S$-equivalence classes of semistable generalized line bundles of degree $d$ or to those of direct images on $X$ of semistable rank $2$ vector bundles of degree $e=d+\deg{\sN}$ on $\Xred$. Observe that $\mathrm{M}(\sO_X,P_0)$ is a natural compactification of the Jacobian of $X$ when the line bundles on $X$ are stable, (i.e. when $g\ge 2\bar{g}$, by \cite[Lemma 3.2]{CK}). 

The main results about its global geometry are the following:
\begin{fact}\label{Fact1} \noindent
\begin{enumerate}
\item\label{Fact1:1} Let $j$ be an integer of the same parity of $d$ such that $0\le j\le -\deg{\sN}-1=g-2\bar{g}$ and let $\bar{Z}_j$ be the closure of the locus of stable generalized line bundles of index $j$. Then $\bar{Z}_j$ is a $g$-dimensional irreducible component of $\mathrm{M}(\sO_X,P_d)$ and any irreducible component containing a stable generalized line bundle is of the form $\bar{Z}_j$ for some $j$.
\item\label{Fact1:2} Let $\mathrm{M}(2,e)$ be the moduli space of semistable rank $2$ vector bundles of degree $e$ on $\Xred$. Then
\begin{enumerate}
\item\label{Fact1:2:1} If $g\le 2\bar{g}-1$ (i.e. $\deg{\sN}\ge 0$), it holds that $\mathrm{M}(\sO_X,P_d)$ is equal to $\mathrm{M}(2,e)$.
\item\label{Fact1:2:2} if $g>2\bar{g}-1$ (i.e. if $\deg{\sN}< 0$), all the irreducible components of $\mathrm{M}(\sO_X,P_d)$ contain a stable generalized line bundle, except at most one, which is $\mathrm{M}(2,e)$, if it exists. If $\bar{g}\le 1$, then $\mathrm{M}(2,e)$ is not an irreducible component of $\mathrm{M}(\sO_X,P_d)$. On the other hand it is so if $\bar{g}\ge 2$ and $g\le 4\bar{g}-3$, i.e. $\deg{\sN}\ge 2-2\bar{g}$.  
\end{enumerate}  
\end{enumerate}
\end{fact}
The first point of the fact is a restatement of \cite[Theorem 4.6]{CK} (observe that we have rescaled the irreducible components there treated to make more apparent what they are: in the notation of \cite{CK} $\bar{Z}_j$ would be $\bar{Z}_{i}$ with $i=[j/2]+1$). The second one is a reformulation of \cite[Theorem 4.7]{CK}. These results had already been stated, although without proof, in \cite[Theorem 3.2(i)]{DEL} in the particular case in which the normal bundle of $\Xred$ in $X$, i.e. $\sN^{-1}$ is the canonical line bundle of $\Xred$. In the same article it was also pointed out that the analysis there developed holds in general for $\sN^{-1}$ of sufficiently large degree.

It is time to state explicitly \cite[Question 4.8]{CK}:
\begin{question}
Assume that $\bar{g}\ge 2$ and $g> 4\bar{g}-3$, or equivalently $\deg{\sN}< 2-2\bar{g}$. Is $\mathrm{M}(2,e)$ an irreducible component of $\mathrm{M}(\sO_X,P_d)$?
\end{question}
The answer is negative. In order to prove that, it is sufficient to show that, under these hypotheses, the direct image of any semistable rank $2$ vector bundle of degree $e$ on $\Xred$ deforms to a generalized line bundle of degree $d$ on $X$.  The following fact collects the ingredients of the proof.
\begin{fact}\label{Fact2} Let $C$ be a smooth curve of genus $\bar{g}$ on $k$ and let $E$ be a rank $2$ vector bundle of degree $e$ on it. Then the following hold:
\begin{enumerate}
\item\label{Fact2:1} $s(E):=e-2\max\{\deg{L}|L\subset E\text{ is a line bundle}\}\le\bar{g}$; equivalently if $M$ is a maximal line subbundle (i.e. a line subbundle of maximal degree) of $E$ then $\deg{M}\ge (e-\bar{g})/2$. In particular if $E$ is semistable (resp. stable) then $(e-\bar{g})/2\le\deg{M}\le e/2$ (resp. $(e-\bar{g})/2\le\deg{M}< e/2$).
\item\label{Fact2:2} For a general semistable $E$, it holds that
\[
s(E)=\left\{ \begin{array}{ll}
\bar{g}-1 & \text{ if $e-\bar{g}$ is odd}\\
\bar{g} &\text{ otherwise.}
\end{array} \right.
\]
In the second case, i.e. $s(E)=\bar{g}$, the maximal line subbundles of $E$ constitute a subscheme of dimension $1$ in $\operatorname{Pic}^{(e-s(E))/2}(C)$.
\item\label{Fact2:3} Let $X$ be a ribbon such that $\Xred=C$, let $\sN$ be its nilradical and let $b$ a non-negative integer of the same parity of $d=e-\deg{\sN}$. The direct image on $X$ of $E$ deforms to a generalized line bundle $\sI$ of degree $d$ and index $b$ if and only if there exists a line bundle $L\subset E$ of degree $(e+\deg{\sN}+b)/2$ and such that $\Hom((E/L)\otimes_{\sO_C}\sN,L)\neq\{0\}$.
\end{enumerate}
\end{fact}
The first point is equivalent to a classical result about ruled surfaces, that had already been proved by C. Segre (obviously over $\mathbb{C}$, see \cite{S}) and had been rediscovered by M. Nagata (see \cite{N}); it can be found in this form in \cite{LN}. The first part of the second one is an immediate consequence of \cite[Proposition 3.1]{LN}, while its second part is \cite[Corollary 4.7]{LN}, although it has originally been proved in the context of ruled surfaces by M. Maruyama (see \cite{M}). The third assertion is a restatement of \cite[Th\'{e}or\`{e}me 7.2.3]{DR} (observe that in the cited article J.-M.  Dr\'{e}zet works over $\mathbb{C}$, but his arguments can be applied to curves over any algebraically closed field).

The following theorem is the main result of the paper.
\begin{thm}\label{Theorem1}
Let $X$ be a ribbon of genus $g$ such that $\bar{g}\ge 2$, where $\bar{g}$ is the genus of $\Xred$. If $g\ge 4\bar{g}-2$ (or equivalently $\deg{\sN}\le 1-2\bar{g}$), then the direct image on $X$ of any rank $2$ vector bundle $E$ of degree $e$ on $\Xred$ deforms to a generalized line bundle $\sI(h)$ on $X$ of degree $d=e-\deg{\sN}$ and of index $b_E(h)=-\deg{\sN}-s(E)-2h$ for any integer $0\le h\le [(-\deg{\sN}-s(E))/2]$. Moreover $E$ is (semi)stable if and only if $\sI(0)$ is (semi)stable and these equivalent conditions imply that $\sI(h)$ is stable for any $1\le h\le [(-\deg{\sN}-s(E))/2]$. 
\end{thm}
\begin{proof}
Let $E$ be a vector bundle as in the hypotheses, set $\delta=\deg{\sN}=2\bar{g}-g-1$ (so under our assumptions it is a negative integer), $s=s(E)$ and $b(h)=b_E(h)=-\delta-s-2h$.

By Fact \ref{Fact2}\ref{Fact2:3} the assertion is equivalent to the fact that $b(h)$ is non-negative and that there exists a line bundle $L(h)\subset E$ of degree $(e+\delta+b(h))/2$ such that $\Hom((E/L(h))\otimes\sN,L(h))\neq\{0\}$.

Let us begin with $h=0$.
Consider $M\subset E$ a maximal line subbundle; by Fact \ref{Fact2}\ref{Fact2:1} it holds that 
\[
\deg{M}=\frac{e+\delta+b(0)}{2}=\frac{e-s(e)}{2}\ge\frac{e-\bar{g}}{2}>\frac{e+\delta}{2},
\]
 where the last inequality holds because $g> 4\bar{g}-2$ is equivalent to $\delta<2-2\bar{g}\le-\bar{g}$ (recall that $\bar{g}\ge 2$). Hence $b(0)$ is a non-negative integer, as desired. 

Thus it is sufficient to show that $\Hom((E/M)\otimes\sN,M)\neq\{0\}$. It is immediate to verify that $M$ is a saturated subsheaf of $E$, being a maximal line subbundle, hence $E/M$ is a line bundle on $\Xred$. So it holds that $\Hom((E/M)\otimes\sN,M)\cong \operatorname{H}^0(\Xred,(E/M)^{-1}\otimes\sN^{-1}\otimes M)$. The latter is surely nontrivial if $\deg{(E/M)^{-1}\otimes\sN^{-1}\otimes M}=-s-\delta\ge \bar{g}$, by Riemann-Roch formula. This is true if either $g> 4\bar{g}-2$ (i.e. $\delta\le -2\bar{g}$) or $g=4\bar{g}-2$ (i.e. $\delta=1-2\bar{g}$) and $s<\bar{g}$. In the first case it holds that $-s-\delta\ge-\bar{g}-\delta\ge\bar{g}$, where the first inequality is due to Fact \ref{Fact2}\ref{Fact2:1}, asserting that $s\le \bar{g}$. In the second case it is immediate that  $-s-\delta>-\bar{g}-\delta\ge\bar{g}-1$.

By Fact \ref{Fact2}\ref{Fact2:1}, it remains only the case $\delta=1-2\bar{g}$ and $s=\bar{g}$, in which $E$ is a general stable rank $2$ vector bundle on $\Xred$. Set $F(M)=(E/M)^{-1}\otimes\sN^{-1}\otimes M$. It holds that $\deg{F(M)}=-s-\delta= \bar{g}-1$, so $h^0(F(M))>0$ if and only if $F(M)$ belongs to the theta-divisor of $\operatorname{Pic}^{\bar{g}-1}(\Xred)$, which is ample. By the last assertion of Fact \ref{Fact2}\ref{Fact2:2} there exists a $1$-dimensional family of maximal line subbundles of $E$ and so the map $M\to F(M)$ defines a curve in $\operatorname{Pic}^{\bar{g}-1}(\Xred)$ (by the fact $F(M)\otimes\sN=M^2\otimes(\det(E))^{-1}$, hence $F(M)\neq F(M')$ for two generic maximal line subbundles $M$ and $M'$). Thus by the ampleness of the theta-divisor the intersection of this curve with it is not zero; hence we can conclude that there exists at least one maximal line subbundle $M$ such that $F(M)$ belongs to it.

Now consider $1\le h\le[(-\delta-s)/2]$. By the previous part it holds that $b(0)=-\delta-s$ is non-negative, therefore also $b(h)=b(0)-2h$ is non-negative.

Let $M$ be, as above, a maximal line subbundle of $E$ such that $\Hom((E/M)\otimes\sN,M)\neq\{0\}$ (in other words such that $(E/M)\otimes\sN$ can be seen as a line subbundle of $M$, because any non-trivial morphism between line bundles is injective) and let $L(h)\subset M$ a line subbundle of degree $(e+\delta+b(h))/2$. It is immediate to verify that the torsion-free part of $E/(L(h))$ is just $E/M$, hence $\Hom((E/L(h))\otimes\sN,L(h))=\Hom((E/M)\otimes\sN,L(h))$. It is clear that there exists an $L(h)\subset M$ such that $\Hom((E/M)\otimes\sN,L(h))$ is not trivial if and only if there exists an $L(h)\subset M$ containing $(E/M)\otimes\sN$; this holds if and only if
\[\deg{(E/M)\otimes\sN}=\frac{e+s+2\delta}{2}\le\deg{L(h)}=\frac{e-s-2h}{2},\]
which is equivalent to $h\le-\delta-s$. Requiring also that $b(h)\ge 0$ we find exactly $1\le h\le [(-\delta-s)/2]$. So for any integer $h$ in the desired range it is possible to find a line bundle $L(h)\subset M\subset E$ with the required properties.
     
The assertion about (semi)stability is almost trivial: $E$ is semistable (resp. stable) if and only if $s(E)\ge 0$ (resp. $>0$) if and only if $b(0)=-\delta- s\le\ -\delta$ (resp. $< -\delta$) if and only if $\sI(0)$ is semistable (resp. stable), where the first equivalence holds by definition and the last one is \cite[Lemma 3.2]{CK}. Moreover $b(h)< b(0)$ for any $1\le h\le[(-\delta-s)/2]$ hence $\sI(h)$ is stable, if $b(0)\le-\delta$, again by \cite[Lemma 3.2]{CK}. 
\end{proof}
\begin{rmk}\label{Rmk:1}
It follows immediately from the proof and from Fact \ref{Fact2}\ref{Fact2:3} that the estimate of the indices in the Theorem is sharp. Indeed, if there were an index $b>b(0)$ for which the theorem were true, one should have a line subbundle $L\subset E$ of degree strictly greater than that of a maximal line subbundle, which is a contradiction. On the other hand, if $b<b([(-\deg{\sN}-s(E))/2])$ and is of its same parity, $b$ is negative and so it cannot be the index of a generalized line bundle. 
\end{rmk}
\begin{cor}\label{Cor:1}
If $\bar{g}\ge 2$ and $g\ge 4\bar{g}-2$, then $\mathrm{M}(2,e)$ is not an irreducible component of $\mathrm{M}(\sO_X,P_d)$. More precisely it is contained in $\bigcap\limits_{\substack{
j=0\\
j\equiv m \pod 2
}}^{m}\bar{Z}_{j}$, with $m=-\deg{\sN}-\bar{g}+1$ if $e-\bar{g}$ is odd and $m=-\deg{\sN}-\bar{g}$ if $e-\bar{g}$ is even. Moreover, it is not contained in any other irreducible component.
\end{cor}
\begin{proof}
The hypotheses are the same of the Theorem, so it follows immediately from it and from Fact \ref{Fact2}\ref{Fact2:2} that $\mathrm{M}(2,e)$ is contained in $\bigcap\limits_{\substack{
j=0\\
j\equiv m \pod 2
}}^{m}\bar{Z}_{j}$. The last assertion is a trivial consequence of the Remark.
\end{proof}

\paragraph*{\textbf{Acknowledgements}}\noindent\\
This paper is born as an aside to my doctoral thesis, which is in progress and is about the compactified Jacobian of a primitive multiple curve of multiplicity greater or equal than $3$. I am grateful to my supervisor, Filippo Viviani, who introduced me to the articles \cite{CK} and \cite{DR} and more generally to the subject and, moreover, gave me various suggestions about the exposition. I am also grateful to Edoardo Sernesi: my knowledge of one of the key ingredients of the proof, namely Segre-Nagata Theorem (i.e. Fact \ref{Fact2}\ref{Fact2:1}) is due to his unpublished notes about algebraic curves which he distributed confidentially in a preliminary version during a doctoral course about Brill-Noether theory.
I would thank also the anonymous referee for his useful comments.\\
A.m.D.g.

\end{document}